\documentclass[11pt]{amsart}
\usepackage{amsrefs}
\usepackage{amsfonts, amsmath, amssymb, amscd, amsthm, bm}
\usepackage{cases}
\usepackage{enumerate}

\usepackage{multicol}
\usepackage{comment}

\usepackage[colorlinks,
linkcolor=blue,
anchorcolor=black,
citecolor=red
]{hyperref}

\newtheorem{thm}{Theorem}[section]
\newtheorem{cor}[thm]{Corollary}

\newtheorem{lem}[thm]{Lemma}
\newtheorem{prop}[thm]{Proposition}
\theoremstyle{definition}

\theoremstyle{theorem}
\newtheorem{rem}[thm]{Remark}

\theoremstyle{claim}
\numberwithin{equation}{section}
\def\F{{\mathbb M}^{n+p}_c}

\def\beric{\Ric_{-}^\lambda}

\DeclareMathOperator\tr{tr}
\DeclareMathOperator\Ric{Ric}

\begin{document}
\title{Rigidity of minimal submanifolds in space forms}

\author{Hang Chen}
\address[Hang Chen]{Department of Applied Mathematics, Northwestern Polytechnical University, Xi' an 710129, P. R. China \\ email: chenhang86@nwpu.edu.cn}
\thanks{Chen supported by NSFC Grant No. 11601426 and by Top International University Visiting Program for Outsanding Young scholars of Northwestern Polytechnical University.}
\author{Guofang Wei}
\address[Guofang Wei]{Department of Mathematics, University of California, Santa Barbara CA 93106 \\ email: wei@math.ucsb.edu}
\thanks{GW partially supported by NSF DMS 1506393}

\begin{abstract}
In this paper, we consider the rigidity for an $n(\geq 4)$-dimensional submanfolds $M^n$ with parallel mean curvature in the  space form $\F$ when the integral Ricci curvature of $M$ has some bound. We prove that, if $c+H^2>0$ and 
$\|\beric\|_{n/2}< \epsilon(n,c, \lambda, H)$
for  $\lambda$ satisfying $ \tfrac{n-2}{n-1} (c+H^2) < \lambda \le c+H^2$, then $M$ is  the totally umbilical sphere $\mathbb S^n(\tfrac{1}{\sqrt{c+H^2}})$. Here $H$ is the norm of the parallel mean curvature of $M$, and $\epsilon(n,c,\lambda, H)$ is a positive constant depending only on $n, c,\lambda$ and $H$. This extends some of the earlier work of \cite{XG13} from pointwise Ricci curvature lower bound to inetgral Ricci curvature lower bound. 
\end{abstract}



\maketitle

\section{Introduction}
There is a long history of studying rigidity phenomenon for submanifolds under certain curvature pinching conditions. A lot of a rigidity theorems for closed minimal submanifolds in a sphere were proved by Simons, Chern-do Carmo-Kobayashi, Lawson, Yau and others  (see \cite{Sim68,Law69,CdCK70,Yau74,Yau75,Ito75,Eji79,Gau86,She89,LL92,She92,Li93}). Let $\mathbb S^n(r)$ and $\F$ denote the $n$-dimensional sphere with radius $r$ and the  $(n+p)$-dimensional (simply-connected) space form with constant curvature $c$ respectively, and we will omit the radius $r$ and just denote $\mathbb S^n$ if $r=1$ for simplicity. In 1979, Ejiri proved the following theorem.

\begin{thm}[\cite{Eji79}]\label{thm:E79}
	Let $M$ be an $n$-dimensional ($n\geq 4$) simply connected compact orientable minimal submanifold immersed in $\mathbb S^{n+p}$. If $\Ric^M\geq n-2$, then $M$ is either the totally geodesic submanifold $\mathbb S^n$, the Clifford torus $\mathbb S^{m}(\sqrt{1/2})\times \mathbb S^{m}(\sqrt{1/2})$ in $\mathbb S^{n+1}$ with $n=2m$, or $\mathbb CP_{4/3}^{2}$ in $\mathbb S^{7}$. Here $\mathbb CP_{4/3}^{2}$ denotes the $2$-dimensional complex projective space minimally immersed in $\mathbb S^7$ with  constant holomorphic sectional curvature $4/3$.
\end{thm}

In 1992, Shen \cite{She92} proved that any $3$-dimensional compact orientable minimal submanifold $M$ immersed in $\mathbb S^{3+p}$ with $\Ric^M >1$ must be totally geodesic. Later, Li \cite{Li93} improved the pinching constant in Ejiri's theorem for odd-dimensional cases. In 2011, Xu and Tian \cite{XT11} pointed out the assumption that $M$ is simply connected in Ejiri's theorem can be removed. In 2013, Xu and Gu proved the following generalized Ejiri rigidity for compact submanifolds with parallel mean curvature in space forms.

\begin{thm}[Theorem 3.3 in \cite{XG13}]
		Let $M$ be an $n$-dimensional ($n\geq 3$)  compact orientable submanifold with parallel mean curvature in the space form $\F$ with $c+H^2> 0$. Here  $H$ is the norm of the parallel mean curvature of $M$.  If $\Ric^M\geq (n-2)(c+H^2)$, then $M$ is either a totally umbilical  sphere  $\mathbb S^n(\frac{1}{\sqrt{c+H^2}})$, the Clifford torus $\mathbb S^{m}(\frac{1}{\sqrt{2(c+H^2})})\times \mathbb S^{m}(\frac{1}{\sqrt{2(c+H^2})})$ in the totally umbilical  sphere  $\mathbb S^{n+1}(\frac{1}{\sqrt{c+H^2}})$ with $n=2m$, or $\mathbb CP_{\frac{4}{3}(c+H^2)}^{2}$ in $\mathbb S^{7}(\frac{1}{\sqrt{c+H^2}})$.
\end{thm}

In particular, this gives 
\begin{cor}[Corollary 3.4 in \cite{XG13}] \label{Cor}
	Let $M$ be an $n (\geq3)$-dimensional oriented compact submanifold with parallel mean curvature in $\F$ with $c + H^2 > 0$. If $Ric^M >(n-2)(c+H^2)$,   then M is the totally umbilical sphere $\mathbb S^n(\frac{1}{\sqrt{c+H^2}})$.
\end{cor}

Note that the curvature conditions in both original and generalized Ejiri theorems are pointwise lower Ricci curvature bounds. It is natural to ask that if we can improve the pinching condition. In odd-dimensional case, the pinching constant can be lowered down (see Li \cite{Li93}, Xu-Leng-Gu \cite{XLG14}'s results). In this paper, we will consider the integral Ricci curvature condition instead of the pointwise Ricci curvature condition.

For each $x\in M$, let $\rho(x)$ be the smallest eigenvalue of the Ricci tensor at $x$, and $\beric(x)=\max\{0,(n-1)\lambda-\rho(x)\}$ for $\lambda \in \mathbb R$. Define
\begin{equation*}
\|\beric\|_{q}:=\left(\int_M(\beric)^q\right)^{1/q},
\end{equation*}
which measures the amount of Ricci curvature lying below the given bound $(n-1)\lambda$. It is easy to see that $\|\beric\|_{q}=0$ if and only if $\Ric^M\geq (n-1)\lambda$.

Now we can state our main theorems.
\begin{thm}\label{thm_min}
	Let $M$ be an $n$-dimensional ($n\geq 4$)  minimal closed  submanifold in $\mathbb S^{n+p}(r)$. Given $\lambda$ satisfying $(n-2)/r^2<(n-1)\lambda\leq (n-1)/r^2$, if
\begin{equation*}
\|\beric\|_{n/2}< \epsilon_r(n,\lambda), 
\end{equation*}
then $M$ is totally geodesic. Here $\epsilon_r(n,\lambda)$
 is an explicit constant defined in \eqref{const-def-epr}.
\end{thm}

In \cite{PW97} Petersen and the second author established  the fundamental  comparison tools, the Laplacian and Bishop-Gromov volume comparisons, for integral Ricci curvature lower bound when $q > \frac n2$. Here we only require smallness of the integral curvatire for $q = \frac n2$ as the manifold is special. 

\begin{rem}
	For a minimal submanifold $M$ in $\mathbb S^{n+p}(r)$, the Ricci curvature of $M$ has the upper bound $(n-1)/r^2$ from \eqref{Ricci} in Section 2. That is why we limit the range of $\lambda$ in Theorem \ref{thm_min}.
\end{rem}

Theorem~\ref{thm_min}
 is a special case of the following result.

\begin{thm}\label{thm_pmc}
	Let $M$ be an $n$-dimensional ($n\geq 4$)  submanifold in $\F$ with parallel mean curvature (PMC). Denote $H$ the norm of the parallel mean curvature of $M$.  Assume $c+H^2>0$. Given $\lambda$ satisfying $(n-2)(c+H^2)<(n-1)\lambda\leq (n-1)(c+H^2)$,
	if 
	\begin{equation*}
	\|\beric\|_{n/2}< \epsilon(n,c,\lambda,H) ,
	\end{equation*}
	then $M$ is  the totally umbilical sphere $\mathbb S^n(\frac{1}{\sqrt{c+H^2}})$. Here $\epsilon(n,c,\lambda,H) $ is an explicit constant defined in \eqref{const-def-ep-pmc}.
\end{thm}

This generalizes Corollary~\ref{Cor} for $n \ge 4$. 

\begin{rem}
	Xu \cite{Xu94} proved that, for an $n(\geq 3)$-dimensional closed $M$ with parallel mean curvature in the unit sphere $\mathbb{S}^{n+p}$, if $\|S-nH^2\|_{n/2}< C(n,p,H)$, then $M$ is $\mathbb S^n(\frac{1}{\sqrt{1+H^2}})$. From \eqref{scalar}, $S-nH^2=n(n-1)(1+H^2)-R$. Here $S$ is the norm of the second fundamental form and $R$ is the scalar curvature. Hence Xu's result is an integral perturbation of scalar curvature while our reult is an integral perturbation  of Ricci curvature. On the other hand while $R \le n(n-1)(1+H^2)$, it is not clear if $\Ric \le (n-1)(1+H^2)$ when $H \not =0$. 
\end{rem}

The paper is organized as follows.
In Section 2 we introduce the notations and recall a few results from \cite{Xu94} which we will need. In Section 3 we prove Theorem~\ref{thm_min}, the starting point is the Simons' identity. In Section 4 we prove Theorem~\ref{thm_pmc} by first showing it is pseudo-umbilical, then reducing it to Theorem~\ref{thm_min} with dimension reduction. 

Acknowledgment:  This work is done while the first author is visitng UCSB. He would like to thank UCSB math department for the hospitality, and he also would like to acknowledge financial support from China Scholarship Council. The authors would like to thank Prof. Haizhong Li for bringing the question to our attention.

\section{Preliminaries}
In this paper, we will use the following convention on the ranges of indices except special declaration:
 \begin{equation*}
1\leq A, B, C,...\leq n+p; \quad 1\leq i, j, k,...\leq n;\quad  n+1\leq\alpha, \beta, \gamma,...\leq n+p.
\end{equation*}

Assume that $M^n$ is immersed in $N^{n+p}$.
We choose a local orthonormal frame $\{e_1,...,e_{n+p}\}$ such that $\{e_1,...,e_{n}\}$ are tangent to $M$ and $\{e_{n+1},...,e_{n+p}\}$ are normal to $M$ when restricted to $M$. Let $\{\omega_A\}$ be the dual coframe. Denote

 \begin{equation*}
 h=\sum_{i,j,\alpha}h_{ij}^{\alpha}\omega_i\otimes\omega_j\otimes
e_{\alpha}
 \end{equation*}
 be the second fundamental form be the mean curvature vector of $M$ immersed in $N$, and define
\begin{align}
A_{\alpha}=&(h_{ij}^{\alpha}), \quad H^{\alpha}=\frac{\tr A_{\alpha}}{n},\quad \mathbf{H}=\sum_{\alpha}H^\alpha e_\alpha,\quad H=|\mathbf{H}|=\sqrt{\sum_{\alpha}(H^{\alpha})^2},\quad S=\sum_{i,j,\alpha}(h_{ij}^{\alpha})^2.
\end{align}
We also denote $[A_\alpha,A_\beta]=A_\alpha A_\beta-A_\beta A_\alpha,\sigma_{\alpha\beta}=\sum_{i,j}h_{ij}^\alpha h_{ij}^\beta,$ and $N(\Omega)=\tr(\Omega^t\Omega)$ the norm of  matrix $\Omega$.

When $\mathbf{H}$ is nowhere zero, we always choose $e_{n+1}=\mathbf{H}/H$, and $\{e_i\}$ diagonalizing $A_{n+1}$, i.e. $h_{ij}^{n+1}=\lambda_i^{n+1}\delta_{ij}.$ Denote
\begin{align}
	S_H&=\sum_{i,j}(h_{ij}^{n+1})^2,\quad \mu_i^{n+1}=H-\lambda_i^{n+1}.\label{def-S_H}
\end{align}

It is well-known that Gauss, Codazzi and Ricci equations are following when $N=\F$:
\begin{align}
&R_{ijkl}=c(\delta_{ik}\delta_{jl}
-\delta_{il}\delta_{jk})+\sum_{\alpha}(h_{ik}^{\alpha}h_{jl}^{\alpha}-h_{il}^{\alpha}h_{jk}^{\alpha}),\label{eqG}\\
&h_{ijk}^{\alpha}=h_{ikj}^{\alpha},\label{eqC}\\
&R_{\alpha\beta
	ij}^{\bot}=\sum_{k}h_{ik}^{\alpha}h_{kj}^{\beta}-\sum_{k}h_{ik}^{\beta}h_{kj}^{\alpha},\label{eqR}
\end{align}
where $R_{ijkl}$ and $h_{ijk}^{\alpha}$ are the components of  Riemannian curvature of $M$ and covariant derivative of $ h_{ij}^{\alpha}$ under the orthonormal frame respectively. The Ricci identity shows that
\begin{equation}\label{eq-Ric-id}
h_{ijkl}^{\alpha}-h_{ijlk}^{\alpha}=\sum_{m}R_{mikl}h_{mj}^{\alpha}+\sum_{m}R_{mjkl}h_{im}^{\alpha}+\sum_{\beta}R_{\beta\alpha kl}^{\bot}h_{ij}^{\beta}
\end{equation}
From (\ref{eqG}), we can get the Ricci curvature and the scalar curvature respectively as following:
\begin{align}
R_{ij}&=c(n-1)\delta_{ij}+n\sum_{\alpha}H^{\alpha}h^{\alpha}_{ij}-\sum_{\alpha, k}h^{\alpha}_{ik}h^{\alpha}_{kj},\label{Ricci}\\
R&=cn(n-1)+n^2H^2-S.\label{scalar}
\end{align}
Since $S \ge nH^2$, we have $R \le n(n-1)(c+H^2)$. When $H =0,\ \Ric \le (n-1)c$. 

Next we recall some results which will be used to prove main theorems. Using a Sobolev inequality in \cite{HS74}, Xu proved the following.
\begin{prop}[cf. \cite{Xu94}]\label{prop-int}
Let $M^n (n\geq 3)$  be a closed submanifold in $N^{n+p}$. Suppose $N$ is a simply connected and complete manifold with non-positive curvature. Then for all $t>0$ and $f\in C^1(M), f\geq 0$, we have
\begin{equation}
 \int_M|\nabla f|^2\geq A(n,t)\left(\int_M f^{\frac{2n}{n-2}}\right)^{\frac{n-2}{n}}-B(n,t)\int_M H^2f^2,
\end{equation}
where
\begin{equation}\label{const-def-ABC}
A(n,t)=\frac{(n-2)^2}{4(n-1)^2(1+t)}\frac{1}{C^2(n)}, 
\quad B(n,t)=\frac{(n-2)^2}{4(n-1)^2t} , \quad C(n)=2^{n}\frac{(n+1)^{1+1/n}}{(n-1)\omega_n^{1/n}},
\end{equation}
and $\omega_n$ is the volume of the unit ball in $\mathbb{R}^n$.
\end{prop}

Now we can prove the following lemma.
\begin{lem}\label{lem-int-sphere}
	Let $M^n (n\geq 3)$  be a closed submanifold in $\F$. Then for all $t>0$ and $f\in C^1(M), f\geq 0$, we have
	\begin{equation}
 \int_M|\nabla f|^2\geq A(n,t)\left(\int_M f^{\frac{2n}{n-2}}\right)^{\frac{n-2}{n}}-B(n,t)\int_M (c_{+}+H^2)f^2,
	\end{equation}  
	where 
\begin{numcases}{c_{+}:=\max\{c, 0\}=}
c, & if $c\geq 0$\label{def-kap1}\\
0, & if $c\leq 0$.\label{def-kap2}
\end{numcases}
\end{lem}
\begin{proof}
	When $c\leq 0$, it is directly from Lemma \ref{prop-int}.
	
	When $c>0$, considering the composition of isometric immersions $M\to \mathbb S^{n+p}(1/\sqrt{c})\to\mathbb{R}^{n+p+1}$, we obtain the conclusion from Lemma \ref{prop-int} (cf. \cite{Xu94,XY11}).	
\end{proof}

We also need the following Kato-type inequalities.
\begin{lem}[Lemma 1 in \cite{Xu94}]\label{lem-kato}
	Let $M^{n}$ be a submanifold with parallel mean curvature in $\mathbb S^{n+p}.$
Set $f_{\epsilon}=(S_{H}-nH^{2}+n\epsilon^{2})^{1/2}$
and $h_{\epsilon}=(S+np\epsilon^{2})^{1/2}$ for any constant $\epsilon \not = 0 \in \mathbb{R}$. 

(i) If $H\neq0$,then
\begin{equation}\label{kato-pmc}
	\sum_{i,j,k}(h_{ijk}^{n+1})^{2}\geq\frac{n+2}{n}|\nabla f_{\epsilon}|^{2},
\end{equation}  
	
(ii) If $H=0$ , then
\begin{equation}\label{kato-min}
|\nabla h|^2=\sum_{\alpha,i,j.k}(h_{ijk}^{\alpha})^{2}\geq\frac{n+2}{n}|\nabla h_{\epsilon}|^{2}.
		\end{equation}  
\end{lem}
\begin{rem}
	In fact, Lemma \ref{lem-kato} remains true when $M$ is a submanifold with parallel mean curvature in $\F$.
	
	Here we require $\epsilon\neq 0$ to make sure the radicands are strictly positive and then we can apply Lemma \ref{lem-int-sphere} to functions $f_\epsilon$ and $h_\epsilon$.
\end{rem}

\section{Minimal Case}
In this section, we prove Theorem \ref{thm_min}.
\begin{proof}[Proof of Theorem \ref{thm_min}]
	 At first, we assume that $r=1$ . Since $\lambda>\frac{n-2}{n-1}$, we can set $\Lambda:=(n-1)\lambda=(n-2)+\delta$ for some $\delta>0$.

	 Gauss equation \eqref{scalar} gives $R=n(n-1)-S$. Since $R \geq n\rho$, we have  
	 \begin{equation}
	 \frac{S-n}{n}\leq(n-2)-\rho.  \label{S-n}
	 \end{equation}
By definition, 	 \begin{equation}
	(n-2)-\rho=-\delta+(\Lambda-\rho)
	 \leq -\delta+\beric. \label{n-2-rho} 
	 \end{equation}
	 
	 Using \eqref{eqG}--\eqref{eq-Ric-id}, after a direct computation, we can obtain the  well-known Simons' identity for a minimal submanifold $M$ in the unit sphere $\mathbb S^{n+p}$ (cf. \cite{Sim68,Eji79}) 
	 \begin{align}\label{eq-Sim}
	 \frac{1}{2}\Delta S=&
	 \sum_{i,j,k,\alpha}(h_{ijk}^\alpha)^2+n\sum_{i,j,\alpha}(h_{ij}^\alpha)^2-\sum_{i,j,k,l,\alpha,\beta}h_{ij}^\alpha h_{ij}^\beta h_{kl}^\alpha h_{kl}^\beta-\sum_{i,j,\alpha,\beta}\Big(\sum_{k}h_{ik}^\alpha h_{kj}^\beta-h_{jk}^\alpha h_{ki}^\beta\Big)^2\\\
	 =&|\nabla h|^2+nS-\sum_{\alpha,\beta}  N([A_\alpha,A_\beta])-\sum_{\alpha,\beta}  \sigma_{\alpha\beta}^2,
	 \end{align}

	 We claim 
	 \begin{align}
	 	\sum_{\alpha,\beta} & N([A_\alpha,A_\beta])\leq 4[(n-1)-\rho]S-\frac{4}{n}\sum_{\alpha}(N(A_\alpha))^2,
	 	\label{eq-clm1}\\
	 	\sum_{\alpha,\beta} & \sigma_{\alpha\beta}^2=\sum_{\alpha}(N(A_\alpha))^2\leq S^2,\label{eq-clm2}
	 \end{align}
	 
	  \eqref{eq-clm2} is obvious, and we use the same argument in \cite{Eji79} to prove \eqref{eq-clm1}.
	  
	 For a fixed $\alpha$, we choose $\{e_i\}$ such that $A_\alpha$ is diagonalized, $A_\alpha=\mbox{diag}\{\lambda_1^\alpha, \cdots, \lambda_n^\alpha\}$, then \eqref{Ricci} gives 
	 \begin{equation}
		\sum_{\beta\neq\alpha, j}(h_{ij}^\beta)^2\leq (n-1)-\rho-(\lambda_i^\alpha)^2 \quad\mbox{for each } i,
	 \end{equation}
	 and
	 \begin{align}
	 \sum_{\beta}N([A_\alpha,A_\beta])=&\sum_{\beta\neq\alpha}N([A_\alpha,A_\beta])=\sum_{\beta\neq\alpha, i,j}(h_{ij}^\beta)^2(\lambda_i^\alpha-\lambda_j^\alpha)^2\\
	 \leq & 2\sum_{\beta\neq\alpha, i,j}(h_{ij}^\beta)^2\Big((\lambda_i^\alpha)^2+(\lambda_j^\alpha)^2\Big)=\sum_{\beta\neq\alpha, i,j}4(h_{ij}^\beta)^2(\lambda_i^\alpha)^2\\
	 \leq & 4\sum_{i}[(n-1)-\rho-(\lambda_i^\alpha)^2](\lambda_i^\alpha)^2.
	 \end{align}
	 Now making summation over $\alpha$, we have 
	 \begin{align}
	 \sum_{\alpha,\beta}  N([A_\alpha,A_\beta])\leq & 4[(n-1)-\rho]S-4\sum_{\alpha}\tr A_\alpha^4, \\ 
	 \leq & 4[(n-1)-\rho]S-\frac{4}{n}\sum_{\alpha}(N(A_\alpha))^2,
	 \end{align}
	 where we used Cauchy-Schwarz inequality in the last inequality, and we complete the proof of the claim.
	 
	 Therefore, from \eqref{eq-Sim},\eqref{eq-clm1} and \eqref{eq-clm2},  we have $\frac{1}{2}\Delta S\geq |\nabla h|^2+ Q$, 
where
	 \begin{align*}
	 	Q:=&S\left[n-4((n-1)-\rho)+\frac{4-n}{n}S\right]\\
	 	=& \frac{n-S}{n} (n-4)S-4\big((n-2)-\rho\big)S\\
	 	\geq &  -(n-4)\big((n-2)-\rho\big)S-4\big((n-2)-\rho\big)S\\
	 	=&  -n\big((n-2)-\rho\big) S. 
	 \end{align*}
	  Here we used \eqref{S-n} for the inequality.

	  From \eqref{n-2-rho}  we have
	  \begin{align}
	  \int_M Q
	  \geq &n\delta\int_{M}S-n\int_{M}\beric S\\
	  \geq& n \delta\int_M S-n \|\beric\|_{n/2}\|S\|_{n/(n-2)},\label{eq-hol}
	  \end{align}
	  here we used H\"{o}lder's inequality in \eqref{eq-hol}.
	  
	 On the other hand, from \eqref{kato-min} and Lemma \ref{lem-int-sphere}, we have
	 \begin{align}
	 	 \int_M|\nabla h|^2\geq&\int_M\frac{n+2}{n}|\nabla h_\epsilon|^2\\
	 	 \geq& \frac{n+2}{n}A(n,t)\|h_\epsilon^2\|_{n/(n-2)}-\frac{n+2}{n}B(n,t)\int_M h_\epsilon^2,
	 \end{align}
where $A(n,t), B(n,t)$ are defined as in \eqref{const-def-ABC}.
Letting $\epsilon\to 0$, we have 
\begin{equation}\label{eq-grad-min}
 \int_M|\nabla h|^2\geq\frac{n+2}{n}A(n,t)\|S\|_{n/(n-2)}-\frac{n+2}{n}B(n,t)\int_M S.
\end{equation}
Then choosing $t$ such that $n\delta=\frac{n+2}{n}B(n,t)$ and from above inequlities, we have 
 \begin{align}
0=&\int_M\frac{1}{2}\Delta S\geq \big(\frac{n+2}{n}A(n,t)-n \|\beric\|_{n/2}\big)\|S\|_{n/(n-2)}\geq 0
\end{align}
provided $\|\beric\|_{n/2}< \frac{n+2}{n^2}A(n,t).$ Hence we have $S\equiv 0$, i.e. $M$ is totally geodesic if we set 
$\epsilon(n,\lambda)=\frac{n+2}{n^2}A(n,t)$.

Now set $\epsilon_r(n,\lambda)=\epsilon(n,\lambda r^2)$, we can prove the theorem for arbitrary $r>0$ by rescaling.
\end{proof}

\begin{rem}\label{rem-const-min}
	In fact,
	\begin{align}
		\epsilon_r(n,\lambda)&=\epsilon(n,\lambda r^2)=\frac{P_n}{1+\frac{1/r^2}{(n-1)\lambda-(n-2)/r^2}P_n}\cdot\frac{1}{C^2(n)}.\label{const-def-epr}
	\end{align}
	where $P_n=\frac{(n+2)(n-2)^2}{4n^2(n-1)^2}$. It is easy to see that $\epsilon(n,\lambda)\to 0^+$ as $\lambda\to (\frac{n-2}{n-1})^+$.
\end{rem}

\section{Parallel Mean Curvature Case}
In this section, we will prove Theorem \ref{thm_pmc}. 
Firstly, we prove the following proposition.

\begin{prop}\label{prop_um}
	Let $M$ be an $n$-dimensional ($n\geq 3$)  submanifold in $\F$ with parallel mean curvature. Assume $c+H^2>0$ and $H\neq 0$. For each 
	$\lambda$ satisfying $(n-2)(c+H^2)<(n-1)\lambda\leq (n-1)(c+H^2)$,
	if 
	\begin{equation}
	\|\beric\|_{n/2}< \epsilon(n,c,\lambda,H), 
	\end{equation}
	then $M$ is pseudo-umbilical.
\end{prop}

\begin{rem}
	We recall that (see \cite{YC71,Chen73}) $M$ is pseudo-umbilical if and only if $\mathbf{H}$ is nowhere zero and the second fundamental form in the direction of $\mathbf{H}$ has the same eigenvalues everywhere, i.e., there exists a function $\lambda$ on $M$ such that
	\begin{equation}
		\langle h(X,Y), \mathbf{H}\rangle=\lambda\langle X,Y\rangle
	\end{equation}
	for all tangent vector field $X, Y$ on $M$. 
\end{rem}

\begin{proof}
	Set $\Lambda:=(n-1)\lambda=(n-2)(c+H^2)+\delta$ for some $\delta>0$.
	From Gauss equation we have 
	\begin{equation}\label{eq-pmc}
	S-nH^2\leq n[(n-1)(c+H^2)-\rho]\leq n[-\delta+(c+H^2)+\beric],
	\end{equation}

Recall the definition of $S_H$ in \eqref{def-S_H}, by some direct computations, we obtain the following estimate (see (3.7) in \cite{XG13}).
\begin{align}
	\frac{1}{2}\Delta S_H\geq & \sum_{i,j,k}(h_{ijk}^{n+1})^2+B_2 Q,
\end{align}
where $B_2=S_H-nH^2=\sum_i(\mu_i^{n+1})^2$ and 	
\begin{equation}
	Q=n(c+H^2)-\frac{n-3}{n-2}(S-nH^2)-\frac{1}{n-2}n[(n-1)(c+H^2)-\rho].
\end{equation}
Now similar as in the proof of Theorem \ref{thm_min}, 
by using \eqref{eq-pmc}, \eqref{kato-pmc} and Lemma \ref{lem-int-sphere}, we have
\begin{align}
	\int_M B_2Q\geq & n\delta \int_M B_2 -\|\beric\|_{n/2}\|B_2\|_{n/(n-2)}.\\
	\int_{M}\sum_{i,j,k}(h_{ijk}^{n+1})^2\geq&\frac{n+2}{n}A(n,t)\|B_2\|_{n/(n-2)}-\frac{n+2}{n}B(n,t)(c_{+}+H^2)\int_M B_2.
\end{align}
Choose $t$ such that $n\delta=\frac{n+2}{n}B(n,t)(c_{+}+H^2)$, then 
\begin{align}
0=&\int_M\frac{1}{2}\Delta S_H\geq \big(\frac{n+2}{n}A(n,t)-n \|\beric\|_{n/2}\big)\|B_2\|_{n/(n-2)}\geq 0
\end{align}
provided $\|\beric\|_{n/2}< \frac{n+2}{n^2}A(n,t).$ 
Hence we have $B_2\equiv 0$, i.e. $M$ is pseudo-umbilical if we set $\epsilon(n,c,\lambda,H)=\frac{n+2}{n^2}A(n,t)$.
\end{proof}

\begin{rem}\label{rem-const-pmc}
In fact,
	\begin{equation}
	\epsilon(n,c,\lambda,H)=\frac{P_n}{1+\frac{c_{+}+H^2}{(n-1)\lambda-(n-2)(c+H^2)}P_n}\cdot\frac{1}{C^2(n)},
\label{const-def-ep-pmc}
	\end{equation}
	where $c_{+}$ is defined as in Lemma \ref{lem-int-sphere}.
	
	We also have
	 $\epsilon(n,c,\lambda,H)=\epsilon_{1/\sqrt{c+H^2}}(n,\lambda)$ for $c\geq0$, and $\epsilon(n,c,\lambda,H)<\epsilon_{1/\sqrt{c+H^2}}(n,\lambda)$ for $c<0$ from Remark \ref{rem-const-min}.
\end{rem}

\begin{proof}[Proof of Theorem \ref{thm_pmc}]
	 The proof is same as the proof of Theorem 3.3 in \cite{XG13} and the following lemma on reduction of codimension due to Yau \cite{Yau74} will be used.
	 
	\begin{lem}[Theorem 1 in \cite{Yau74}]\label{lem-redu}
		Let $N$ be a conformally flat manifold. Let $N_1$ be a sub-bundle of the normal bundle of $M$ with fiber dimension $k$. Suppose $M$ is umbilical with respect to $N_1$ and $N_1$ is parallel in the normal bundle. Then $M$ lies in an $n+p-k$ dimensional umbilical submanifold $N'$ of $N$ such that the fiber of $N_1$ is everywhere perpendicular to $N'$. 
	\end{lem}
	 
	 When $H=0$, that is Theorem \ref{thm_min}.
	 
	 When $H\neq 0$. If $p=1$, then the conclusion is from Proposition \ref{prop_um}. 
	 
	 If $p\geq 2$, we can conclude $M$ is a minimal submanifold in $\mathbb S^{n+p-1}(\frac{1}{\sqrt{c+H^2}})$. The detail can be found in \cite{XG13}, but we restate it briefly for convenience of the reader. From Lemma \ref{lem-redu}, $M$ actually lies in 
	 $\mathbb{M}^{n+p-1}_{\tilde{c}}$. Then $\mathbf{H}$ is decomposed orthogonally into two parts
	 \begin{equation}
	 \mathbf{H}=\mathbf{H}_1+\mathbf{H}_2,
	 \end{equation}
	 where $\mathbf{H}_1$ is the mean curvature  of $M$ in $\mathbb{M}^{n+p-1}_{\tilde{c}}$, and $\mathbf{H}_2$ is normal to $\mathbb{M}^{n+p-1}_{\tilde{c}}$ in $\F$. But $\mathbf{H}\perp \mathbf{H}_1$ from Lemma \ref{lem-redu} again, we have $\mathbf{H}_1=0$, which means $M$ is minimal in $\mathbb{M}^{n+p-1}_{\tilde{c}}$, and $H=|\mathbf{H}|=|\mathbf{H}_2|$, which derives  $\tilde{c}^2=c^2+H^2$ from Gauss equation.
		 
	 Since $\epsilon(n,c,\lambda,H)\leq\epsilon_{1/\sqrt{c+H^2}}(n,\lambda)$ from Remark \ref{rem-const-pmc}, we complete the proof by applying Theorem \ref{thm_min}. 
\end{proof}


\end{document}